\documentclass[12pt,a4paper]{amsart}
\usepackage[top=1.15in, bottom=1.15in, left=1.17in, right=1.17in]{geometry}
\usepackage{amssymb}
\usepackage{pdflscape}
\usepackage{amscd}
\usepackage{fancyhdr}
\input xy 
\xyoption{all}
\usepackage{amsmath,amsfonts,stmaryrd}
\usepackage[usenames,dvipsnames,svgnames]{xcolor}
\usepackage[colorlinks, urlcolor=Navy, linkcolor=Navy, citecolor=Navy]{hyperref}
\usepackage[abbrev,msc-links]{amsrefs}
\usepackage[latin1]{inputenc}
\usepackage{tikz}
\usetikzlibrary{shapes,arrows}
\theoremstyle{plain}

\newtheorem{thm}{Theorem}[section]
\newtheorem{lem}[thm]{Lemma}

\newtheorem{cor}[thm]{Corollary}

\theoremstyle{definition}
\newtheorem{dfn}[thm]{Definition}

\newtheorem{rem}[thm]{Remark}

\DeclareMathOperator{\Hom}{Hom}

\DeclareMathOperator{\End}{End}
\DeclareMathOperator{\Ext}{Ext}
\DeclareMathOperator{\Tor}{Tor}

\DeclareMathOperator{\modu}{mod}

\DeclareMathOperator{\Modu}{Mod}

\newcommand{\op}{\mathrm{op}}

\DeclareMathOperator{\id}{id}

\newcommand{\oTo}{\xymatrix{ \ar@{^{(}->}[r]|{\mathbf{O}}& }} % open immersion
\newcommand{\cTo}{\xymatrix{ \ar@{^{(}->}[r]|{\mathbf{|}}& }} % closed immersion
\newcommand{\coTo}{\xymatrix{ \ar@{^{(}->}[r]|{\mathbf{O}}|{\mathbf{|}}& }} %locally closed

\DeclareMathOperator{\Ker}{ker}
\DeclareMathOperator{\coKer}{coker}
\DeclareMathOperator{\Bild}{Im}

\DeclareMathOperator{\cogen}{cogen}
\DeclareMathOperator{\gen}{gen}

\newcommand{\La}{\Lambda}

\newcommand{\mcC}{\mathcal{C}}
\newcommand{\mcD}{\mathcal{D}}
\newcommand{\mcE}{\mathcal{E}}
\newcommand{\mcF}{\mathcal{F}}
\newcommand{\mcG}{\mathcal{G}}

\newcommand{\mcI}{\mathcal{I}}

\newcommand{\mcM}{\mathcal{M}}

\newcommand{\mcP}{\mathcal{P}}

\newcommand{\mcX}{\mathcal{X}}

\DeclareMathOperator{\add}{add}

\usepackage{todonotes}

\begin{document}

\title{On faithfully balancedness in functor categories}
\author{Julia Sauter}
\address{Julia Sauter\\ Faculty of Mathematics \\
Bielefeld University \\
PO Box 100 131\\
D-33501 Bielefeld }
\email{jsauter@math.uni-bielefeld.de}

\begin{abstract}
This is a generalization of some results of Ma-Sauter from module categories over artin algebras to more general functor categories (and partly to exact categories). In particular, we generalize the definition of a faithfully balanced module to a \emph{faithfully balanced subcategory} and find the generalizations of dualities and characterizations from Ma-Sauter. 
\end{abstract}

%16G10, 18E30, 18G05}
%tilting, dominant dimension, higher Auslander--Reiten theory, recollement}

\date{\today}
\maketitle
\section{Introduction}

For an exact category $\mcE$ in the sense of Quillen and a full subcategory $\mcM$
we define categories $\gen_k^{\mcE} (\mcM)$ (and $\cogen^k_{\mcE} (\mcM)$) of $\mcE$ (consisting of objects admitting a certain $k$-presentation in $\mcM$). We also consider the two functors $\Phi (X):= \Hom_{\mcE}(-,X)\lvert_{\mcM}, \Psi (X) := \Hom_{\mcE} (X,-)\lvert_{\mcM}$.

We give the relatively obvious but technical generalizations of results in \cite{MS} related to these categories and functors. If $\mcE$ is a functor category (of some sort) these functors have adjoints and therefore stronger results can be found. We state here two of these: \\
Let $\mcP$ be an essentially small additive category. We denote by $\Modu -\mcP$ the category of contravariant additive functors $\mcP \to (Ab)$ (and we set $\mcP-\Modu:= \Modu-\mcP^{op}$). We write $\modu_k-\mcP$ for the full subcategory which admit a $k$-presentation by finitely generated projectives. We denote by $h \colon \mcP \to \Modu-\mcP$, $P \mapsto h_P =\Hom_{\mcP}(-,P)$ the Yoneda embedding. \\
\emph{Cogen$^1$-duality:} Let $k \in \mathbb{N}_0 \cup \{\infty\}$ and assume now $\mcM \subset \modu_k-\mcP$. We shorten the notation $\cogen^k(\mcM):= \cogen^k_{\modu_k-\mcP}(\mcM) \subset \modu_k-\mcP$. \\
We say $\mcM$ is \textbf{faithfully balanced} if $h_P\in \cogen^1(\mcM)$ for all $P \in \mcP$.

\begin{lem} (cf. Lem. \ref{cogen1Duality}) ($\cogen^1$-duality) If $\mcM$ is faithfully balanced, we denote by $\tilde{\mcM}=\Psi (h_{\mcP})\subset \mcM-\modu_k$, then $\Psi $ defines a contravariant equivalence 
\[ 
\cogen^1_{\modu_1-\mcP}(\mcM) \longleftrightarrow \cogen^1_{\mcM-\modu_1}(\tilde{\mcM})
\]
%and contravariant equivalences 
%\[ 
%\cogen^k_{\modu_k-\mcP}(\mcM) \longleftrightarrow \cogen^1_{\mcM-\modu_1}(\tilde{\mcM}) \cap \bigcap_{1\leq i <k} \Ker (\Ext^i_{\mcM-\modu_k} (-, \tilde{\mcM} ))
%\]
\end{lem}

\emph{The symmetry principle} states as follows: 
\begin{thm} (cf. Thm. \ref{sym}, Symmetry principle). 
Let $\mcE$ be an exact category with enough projectives $\mcP$ and enough injectives $\mcI$ and $k \geq 1$. The following two statements are equivalent: 
\begin{itemize}
    \item[(1)] $\mcP \subset \cogen^k_{\mcE} (\mcM)$ and $\Phi(I)=\Hom_{\mcE} (-,I)\lvert_{\mcM} \quad \in \modu_k-\mcM$ for every $I \in \mcI$ 
    \item[(2)] $\mcI \subset \gen_k^{\mcE} (\mcM)$ and $\Psi(P)=\Hom_{\mcE} (P,-)\lvert_{\mcM}\quad \in \mcM-\modu_k$ for every $P \in \mcP$
\end{itemize}
\end{thm}

A nice special case: Assume additionally that $\mcE$ is a Hom-finite $K$-category for a field $K$ and $\mcM = \add (M)$ for an object $M \in \mcE$. Then the following two statements are equivalent: 
\begin{itemize}
    \item[(1)] $\mcP \subset \cogen^k_{\mcE} (\mcM)$  
    \item[(2)] $\mcI \subset \gen_k^{\mcE} (\mcM)$ 
\end{itemize}

Since: If we set $\La=\End_{\mcE}(M)$, then $\modu_k-\mcM$, $\mcM-\modu_k$ can be identified with finite-dimensional (left and right) modules over $\La$ and $\Phi (I)=\Hom_{\mcE} (M,I), \Psi (P) = \Hom_{\mcE} (P,M)$ are by assumption finite-dimensional $\La$-modules.

%\section{On subcategories (co)generated by other subcategories}
\section{In additive categories}
Here we want to extend Yoneda's embedding to a bigger subcategory: 
Let $\mcC$ be an additive category and $\mcM$ an essentially small  full additive subcategory. A right $\mcM$-module is a contravariant additive functor from $\mcM$ into abelian groups. We denote by $\Modu-\mcM$ the category of all right $\mcM$-modules. This is an abelian category. We have the fully faithful (covariant) Yoneda embedding $\mcM \to \Modu-\mcM$ defined by $M \mapsto \Hom_{\mcM}(-,M)$. Clearly, we can extend this functor to a functor $\Phi \colon \mcC \to \Modu-\mcM, \; \Phi (X):= \Hom_{\mcC}(-,X)\lvert_{\mcM} =(-,X)\lvert_{\mcM}$ where the last notation is our shortage for the Hom functor. 
The aim of this section is to define a subcategory $\mcM \subset \mcG \subset \mcC$ such that $\Phi\lvert_{\mcG}$ is fully faithful. \\
We define a full subcategory of $\mcC$ as follows
\[ 
\gen_1^{\add} (\mcM):= \Bigg\{ Z \in \mcC\mid 
\begin{aligned}
&\exists M_1 \xrightarrow{f} M_0 \xrightarrow{g}Z, \; M_i \in \mcM, \; g=\coKer (f) \text{ is an epim. } \\
&(M,M_1) \to (M,M_0) \to (M,Z) \to 0 \text{ ex. seq. of ab. groups }\forall M\in \mcM
\end{aligned}
\Bigg\}
\]
We observe that $g=\coKer (f)$ and $g$ an epimorphism is equivalent to  that we have an exact sequence of $\mcC^{op}$-modules 
\[ 
0 \to (Z,-) \to (M_0,-) \to (M_1, -)
\]
Furthermore the second line in the definition is equivalent to an exact sequence in $\Modu-\mcM$ 
\[ (-,M_1) \to (-,M_0) \to (-,Z) \lvert_{\mcM} \to 0.\]
Dually, we define $\cogen_{\add}^1 (\mcM):= (\gen_1^{\add}(\mcM^{op}))^{op}$ where $\mcM^{op}$ is considered as a full additive subcategory of $\mcC^{op}$. 

\begin{lem}\label{extYon}
\begin{itemize}
    \item[(1)] The functor $\gen_1^{\add} (\mcM) \to \Modu-\mcM$ defined by $Z\mapsto (-,Z)\lvert_{\mcM}$ is fully faithful. We even have for every $Z \in \gen_1^{\add}(\mcM), C \in \mcC$ a natural isomorphism 
    \[ 
    \Hom_{\mcC} (Z, C) \to \Hom_{\Modu-\mcM} ( (-,Z)\lvert_{\mcM}, (-,C)\lvert_{\mcM})
    \]
    
    \item[(2)] The functor $\cogen^1_{\add}(\mcM) \to \Modu-\mcM^{op}$ defined by $Z\mapsto  (Z,-)\lvert_{\mcM}$ is fully faithful. 
     We even have for every $Z\in \cogen^1_{\add}(\mcM ), C \in \mcC$ a natural isomorphism 
    \[ 
    \Hom_{\mcC} (C, Z) \to \Hom_{\Modu-\mcM^{op}} ( (Z,-)\lvert_{\mcM}, (C,-)\lvert_{\mcM})
    \]
\end{itemize}
\end{lem}

\begin{proof}
We only prove (1), the second statement follows by passing to opposite categories. We consider the functor $\Phi \colon \mcC \to \Modu-\mcM$ defined by $\Phi (X):=(-,X)\lvert_{\mcM}$. Since $Z\in \gen_1^{\add}(\mcM)$ we an exact sequences  
\[ 
0 \to (Z,C) \to (M_0,C)\to (M_1, C) \quad \text{of ab. groups}
\]
and $\Phi (M_1) \to \Phi (M_0) \to \Phi (Z) \to 0$ in $\Modu-\mcM$. 
By applying $(-, \Phi (C))$ to the second exact sequence we obtain an exact sequence 
\[ 
0 \to (\Phi (Z),\Phi ( C)) \to (\Phi (M_0) ,\Phi (C))\to (\Phi (M_1), \Phi (C)) \quad \text{of ab. groups.}
\]
Since $\Phi$ is a functor, we find a commuting diagram 
\[ 
\xymatrix{
0 \ar[r]& (Z,C) \ar[r]\ar[d]& (M_0,C)\ar[r]\ar[d]& (M_1, C)\ar[d] \\
0 \ar[r]& (\Phi (Z),\Phi ( C)) \ar[r]& (\Phi (M_0) ,\Phi (C))\ar[r]& (\Phi (M_1), \Phi (C))
}
\]
By the Lemma of Yoneda, we have for every $F \in \Modu-\mcM$ and $M \in \mcM$ that $\Hom_{\Modu-\mcM}(\Phi (M), F)=F(M)$. This implies that the maps $(M_i, C) \to (\Phi (M_i), \Phi (C))$ are isomorphisms of groups. and therefore, the induced map on the kernels is an isomorphism. 
\end{proof}

\begin{rem} \label{remNotSmall}
If $\mcM$ is not essentially small, $\Hom_{\mcM-\Modu}(F,G)$ is not necessarily a set. But if one passes to the full subcategory of finitely presented $\mcM$-modules $\modu_1-\mcM$, this set-theoretic issue does not arise: 
Observe that $Z\mapsto (-,Z)\lvert_{\mcM}$ defines by definition a covariant functor 
\[ 
\Phi\colon \gen_1^{\add}(\mcM) \to \modu_1-\mcM, 
\]
the same proof as before shows that this is fully faithful. Similarly, the functor $Z \mapsto (Z,-)\lvert_{\mcM}$ defines a fully faithful contravariant functor 
\[ \Psi \colon \cogen^1_{\add}(\mcM) \to \modu_1-\mcM^{op}. \] 
\end{rem}

\section{In exact categories}

This section is a generalization of results from \cite{MS}. 
For exact categories we have subcategories of $\cogen^1_{\add}$ such that $\Psi$ induces isomorphisms on (some) extension groups (cf. Lemma \ref{isoOnExt}).\\
Given an exact category $\mcE$ with a full additive subcategory $\mcM$, we define $\cogen^k_{\mcE} (\mcM) \subset \mcE$ to be the full subcategory of all objects $X$ such that there is an exact sequence 
\[ 
0 \to X \to M_0 \to \cdots \to M_k\to Z \to 0
\]
with $M_i \in \mcM, 0\leq i \leq k$ such that for every $M \in \mcM$ the sequence 
\[ 
\Hom_{\mcE} (M_k,M) \to \cdots \to \Hom_{\mcE} (M_0,M) \to \Hom_{\mcE}(X,M) \to 0
\]
is an exact sequence of abelian groups. \\
We define $\gen_k^{\mcE}(\mcM)$ to be the full additive category of $\mcE$ given by all $X$ such that there is an exact sequence \[ 
0 \to Z \to M_k \to \cdots \to M_0 \to X \to 0
\]
with $M_i \in \mcM, 0\leq i\leq k$ such that for every $M \in \mcM$ we have an exact sequence 
\[ 
\Hom_{\mcE}(M,M_k) \to \cdots \to \Hom_{\mcE}(M, M_0) \to \Hom_{\mcE} (M, X) \to 0
\]
of abelian groups. \\
If it is clear from the context in which exact category we are working, then we leave out the index $\mcE$ and just write $\cogen^k(\mcM)$ and $\gen_k(\mcM)$. 
\begin{rem} \label{ffcogenk}
Observe that $\cogen^k_{\mcE}(\mcM)\subset \cogen^1_{\add}(\mcM)$, $\gen_k^{\mcE} (\mcM) \subset \gen_1^{\add} (\mcM)$ for $k\geq 1$ and therefore the functor $\Psi\colon X \mapsto (X,-)\lvert_{\mcM}$ (resp. $\Phi \colon X \mapsto (-,X)\lvert_{\mcM}$) is fully faithful on $\cogen^k_{\mcE}(\mcM)$ (resp. on $\gen_k^{\mcE} (\mcM)$) by Lemma \ref{extYon} and Remark \ref{remNotSmall}. 
\end{rem}

\begin{rem} Let $k \geq 1$. We denote by $\modu_k-\mcM$ the category of $\mcM$-modules  which admit a $k$-presentation (indexed from $0$ to $k$) by finitely presented projectives. For $F \in \modu_k-\mcM$ the Ext-groups $\Ext^i_{\mcM-\Modu}(F,G)$ with $0\leq i < k$ are sets. \\
If $X\in \cogen^k_{\mcE}(\mcM)$, then we have $\Psi (X)= (X,-)\lvert_{\mcM}\;\in \modu_k-\mcM^{op} (=: \mcM-\modu_k)$. \\
If $Y \in \gen_k^{\mcE}(\mcM)$, then we have $\Phi (Y) =(-,Y)\lvert_{\mcM} \; \in \modu_k-\mcM$.
\end{rem}

Since we are now working in exact categories, we observe the following isomorphisms on extension groups:
\begin{lem} \label{isoOnExt}
Let $k\geq 1$. 
\begin{itemize}
    \item[(a)] If $X \in \cogen^k_{\mcE} (\mcM)$, then the functor $Z \mapsto \Psi (Z)=(Z,-)\lvert_{\mcM}$ induces a well-defined natural isomorphism of abelian groups
    \[ 
    \Ext^i_{\mcE} (Y,X) \to \Ext^i_{\mcM-\Modu} ( \Psi (X), \Psi (Y)), \quad 0 \leq i <k 
    \]
    for all $Y \in \bigcap_{1\leq i <k} \Ker \Ext^i_{\mcE} (-, \mcM)$.
     \item[(b)] If $Y \in \gen_k^{\mcE} (\mcM)$, then the functor $Z \mapsto \Phi (Z)=(-,Z)\lvert_{\mcM}$ induces a well-defined natural isomorphism of abelian groups 
    \[ 
    \Ext^i_{\mcE} (Y,X) \to \Ext^i_{\Modu - \mcM} ( \Phi (Y), \Phi (X)), \quad 0 \leq i <k 
    \]
    for all $X \in \bigcap_{1\leq i <k} \Ker \Ext^i_{\mcE} ( \mcM , -)$.
\end{itemize}

\end{lem}
\begin{proof}
(a) the proof is a straight forward generalization of \cite{MS}, Lemma 2.4, (2) (using Rem. \ref{ffcogenk}) and (b) follows from (a) by passing to the opposite exact category $\mcE^{op}$. 
\end{proof}

We will later use the following simple observation:

\begin{rem} \label{gen-deflationClosed}
Let $\mcE$ be an exact category, $\mcX$ be a fully exact category and $\mcM \subset \mcX$ an additive subcategory. We say $\mcX$ is \emph{deflation-closed} if for any deflation $d\colon X\to X'$ in $\mcE$ with $X,X'$ in $\mcX$ it follows $\Ker d\in \mcX$. The dual notion is \emph{inflation-closed}.  \\
If $\mcX$ is deflation-closed then $\gen_k^{\mcX}(\mcM)= \gen_k^{\mcE} (\mcM) \cap \mcX$. If $\mcX$ is inflation-closed then $\cogen^k_{\mcX} (\mcM) =\cogen^k_{\mcE} (\mcM) \cap \mcX$. 
\end{rem}

%%%%%%%%%%%%%%%%%%%%%%%%%%%%%%%%%%%%%%%%%%%%%%%%%%%%%%%%%%%
%%%%%%%%%%%%%%%%%%%%%%%%%%%%%%%%%%%%%%%%%%%%%%%%%%%%%%%%%%%
\subsection{Inside functor categories}
Let $\mcP$ be an essentially small additive category.
We denote by $h \colon \mcP \to \Modu-\mcP$, $P \mapsto h_P =\Hom_{\mcP}(-,P)$ the Yoneda embedding, we write $h_{\mcP}$ for the essential image of $h$. 

\subsubsection{Adjoint functors}
Let now $\mcM$ be an essentially small full additive subcategory of $\Modu-\mcP$. We consider the contravariant functor
\[
\begin{aligned}
\Psi \colon \Modu-\mcP &\to \mcM - \Modu, \\
X &\mapsto \Hom_{\Modu- \mcP} (X,-)\lvert_{\mcM}=(X,-)\lvert_{\mcM}
\end{aligned}
\]
We also consider the contravariant functor 
\[ 
\begin{aligned}
\Psi' \colon \mcM-\Modu &\to \Modu-\mcP\\
Z & \mapsto (P \mapsto \Hom_{\mcM-\Modu} (Z,\Psi (h_P)))  
\end{aligned}
\]

We generalize \cite{ASoII}, Lem. 3.3..

\begin{lem}
The functors $\Psi$ and $\Psi'$ are contravariant adjoint functors, i.e. the following is a (bi)natural isomorphim
\[
\chi\colon \Hom_{\Modu-\mcP} (X, \Psi'(Z)) \to \Hom_{\mcM-\Modu}(Z, \Psi (X))
\]
defined as follows: A natural transformation $f\in \Hom_{\Modu - \mcP}(X, \Psi'(Z))$, is determined by for every $P \in \mcP, x\in X(P), M \in \mcM$ a group homomorphism 
\[ 
f_{P,x}(M) \colon Z(M) \mapsto \Psi (h_P) (M) = M(P)
\]
then, we define a natural transformation $\chi (f) \colon Z \to \Psi (X)=\Hom_{\Modu-\mcP} (X,-)\lvert_{\mcM}$ for $M \in \mcM$ as follows 
\[ 
\begin{aligned} 
\chi (f) (M) \colon Z(M) &\to \Hom_{\Modu-\mcP} (X,M), \\
z &\mapsto (X(P) \xrightarrow{f_{P,-}(z)} M(P), x \mapsto f_{P,x}(M)(z))_{P\in \mcP} 
\end{aligned}
\]
\end{lem}

\begin{proof}
We define $\chi' \colon  \Hom_{\mcM-\Modu}(Z, \Psi (X))\to  \Hom_{\Modu-\mcP} (X, \Psi'(Z)) $ as follows: For $g\colon Z \to \Psi (X)= \Hom_{\Modu-\mcP}(X,-)\lvert_{\mcM}$ we have for every $M \in \mcM, z \in Z(M)$ a natural transformation   
$g_{M,z} \colon X \to M$, i.e. for every $P \in \mcP$ a group homomorphism  
\[ 
g_{M,z}(P) \colon 
X(P) \to M(P), x \mapsto g_{M,z} (P) (x),
\]
then we define $\chi' (g) (P) \colon X (P) \to \Psi' (Z) (P) = \Hom_{\mcM-\Modu}(Z, (h_P,-)\lvert_{\mcM})$ as follows 
\[ x \mapsto (Z(M) \to M(P), z \mapsto g_{M,z}(P) (x))_{M\in \mcM}. \] 
Then $\chi'$ is the inverse map to $\chi$. 
\end{proof}

\begin{rem} \label{contrAdj}
Given an adjoint pair of contravariant functors $\Psi$ and $\Psi'$, the natural isomorphisms 
\[ 
\Hom (X, \Psi (Z) ) \to \Hom (Z, \Psi ' (X))
\]
induce natural transformations $\alpha \colon \id \to \Psi' \Psi$ (and $\alpha' \colon \id \to \Psi \Psi'$) as follows 
\[ 
\Hom (X,X) \xrightarrow{\Psi(-)} \Hom (\Psi (X),\Psi (X)) \cong \Hom (X, \Psi'\Psi (X)), \quad  \id_X \mapsto \alpha_X
\]
in this case we have triangle identities  
\[ 
\begin{aligned}
\id_{\Psi(X)} &= (\Psi (X) \xrightarrow{\alpha'_{\Psi(X)}}
  \Psi \Psi'\Psi (X) \xrightarrow{\Psi(\alpha_X)} \Psi (X))\\
  \id_{\Psi'(Z)} &= (\Psi' (Z) \xrightarrow{\alpha_{\Psi'(Z)}}
  \Psi' \Psi\Psi' (Z) \xrightarrow{\Psi'(\alpha'_Z)} \Psi' (Z))
\end{aligned}
\] 
\end{rem}

In \cite{Y}, section 4, a tensor bifunctor is introduced
\[ 
-\otimes_{\mcM}- \colon \Modu -\mcM \times \mcM- \Modu \to (Ab), (F,G) \mapsto F\otimes_{\mcM}G
\]
Now, we consider the covariant funtor
\[ 
\Phi \colon \Modu-\mcP \to \Modu-\mcM, \quad X \mapsto \Hom_{\Modu-\mcP}(-,X)\lvert_{\mcM} =:(-,X)\lvert_{\mcM}
\]
and the following covariant functor 
\[ 
\Phi' \colon \Modu-\mcM \to \Modu-\mcP, \quad Z \mapsto 
(P \mapsto Z\otimes_{\mcM} \Psi (h_P) )
\]

\begin{lem}
The functor $\Phi$ is right adjoint to $\Phi'$, i.e. we have a (bi)natural maps 
\[ 
\Hom_{\Modu-\mcP} (\Phi'(Z), X)  \to \Hom_{\Modu-\mcM}(Z, \Phi (X))
\]
\end{lem}

\begin{rem}
If $F \colon \mcC \leftrightarrow \mcD \colon G$ is an adjoint pair of functors (with $F$ left adjoint to $G$), then we have a unit $u \colon 1_{\mcC} \to GF$ and a counit, $ c \colon FG \to 1_{\mcD}$. 
Let $\mcC_u$ be the full subcategory of objects in $X$ in $\mcC$ such that $u(X)$ is an isomorphism. Let $\mcD_c$ be the full subcategory of objects $Y$ in $\mcD$ such that $c(Y)$ is an isomorphism. 
Then, the triangle identities show directly that $F,G$ restrict to quasi-inverse equivalences $F\colon \mcC_u \leftrightarrow \mcD_c\colon G$. 
\end{rem}

\subsubsection{$\boxed{\cogen^k}$}
Let $k \in \mathbb{N}_0 \cup \{\infty\}$ and assume now $\mcM \subset \modu_k-\mcP$. 
In this subsection we study $\cogen^k(\mcM):= \cogen^k_{\modu_k-\mcP}(\mcM) \subset \modu_k-\mcP$. \\
Our aim is to give a different description of the categories $\cogen^k(\mcM)$ (cf. Lemma \ref{diffdesc})  and to introduce \emph{faithfully balancedness} which leads to the $\cogen^1$ duality (cf. Lemma \ref{cogen1Duality}). 
%%%%%%%%%%%%%%%%%%%%%%%%%%%%%%%%%%%%%%%%%%

We have the contravariant functor
\[ 
\Psi \colon \Modu-\mcP \to  \mcM-\Modu, \quad X \mapsto \Hom_{\Modu-\mcP}(X,-)\lvert_{\mcM}
\]
and  $\Psi \lvert_{\cogen^k(\mcM )}\colon \cogen^k(\mcM) \to \mcM-\modu_k$ is fully faithful for $1 \leq k <\infty$.\\

%We have an induced covariant functor 
%\[ 
%\eta =\Psi'\circ \Psi \colon \Modu-\mcP \to \Modu-\mcP, \quad X \mapsto \eta_X
%\]
%defined for $P \in \mcP$ as 
%\[\eta_X (P)= \Hom_{\mcM -\Modu}(\Psi (X), \Psi(h_P))\]
The natural transformation $\alpha \colon \id_{\Modu-\mcP} \to \Psi'\Psi $, for $X \in \Modu-\mcP$ is given by a morphism in $\Modu-\mcP$, $\alpha_X\colon X \to \Psi'\Psi (X)= \Hom_{\mcM -\Modu}(\Psi (X), \Psi(h_{-}))$ which is defined at $P \in \mcP$ via 
\[ 
\begin{aligned}
X(P)=\Hom_{\Modu-\mcP}(h_P, X)   &\to \Hom_{\mcM-\Modu} (\Hom_{\Modu-\mcP} (X,-)\lvert_{\mcM}, \Hom_{\Modu-\mcP} (h_P,-)\lvert_{\mcM}) \\
f  &\mapsto  [\Hom_{\Modu-\mcP}(X, -) \xrightarrow{-\circ f} \Hom_{\Modu-\mcP} (h_P , -) ]\lvert_{\mcM}
\end{aligned}
\]

We observe that $\alpha_M$ is an isomorphism for every $M \in \mcM$ (since \[ (\Psi'\Psi (M)) (P)= \Hom_{\mcM-\Modu}(\Hom_{\mcM}(M,-), \Psi (h_P)) = \Psi (h_P) (M) = \Hom_{\Modu-\mcP} (h_P, M)= M(P)\] using Yoneda's Lemma twice).

\begin{lem} \label{diffdesc}
For $1\leq k \leq \infty$ we have 
\[
\begin{aligned}
&\cogen^k_{\modu_k-\mcP}(\mcM) = \\
&\{ X\in \modu_{k}-\mcP \mid \alpha_X \text{ isom. },\Psi (X) \in \mcM-\modu_k, \Ext^i_{\mcM -\Modu}(\Psi (X), \Psi (h_P))=0, 1\leq i < k, \forall P \in \mcP \}
\end{aligned}
\]
\end{lem}

\begin{proof}
The proof is a straight forward generalization of \cite{MS}, Lemma 2.2, (1) (the functor $\Hom_{\Gamma} (-,M)$ has to be replaced by applying $\Hom_{\mcM-\Modu} ( -, \Psi (h_P))$ for all $P \in \mcP$). 
\end{proof}

\begin{dfn} We say $\mcM$ is \textbf{faithfully balanced} if $h_{\mcP}\subset \cogen^1(\mcM)$. 
\end{dfn}

\begin{lem} \label{cogen1Duality} ($\cogen^1$ duality) If $\mcM$ is faithfully balanced, we denote by $\tilde{\mcM}=\Psi (h_{\mcP})\subset \mcM-\modu_k$, then $\Psi $ defines a contravariant equivalence 
\[ 
\cogen^1_{\modu_1-\mcP}(\mcM) \longleftrightarrow \cogen^1_{\mcM-\modu_1}(\tilde{\mcM})
\]
and contravariant equivalences 
\[ 
\cogen^k_{\modu_k-\mcP}(\mcM) \longleftrightarrow \cogen^1_{\mcM-\modu_1}(\tilde{\mcM}) \cap \bigcap_{1\leq i <k} \Ker (\Ext^i_{\mcM-\modu_k} (-, \tilde{\mcM} ))
\]
\end{lem}

\begin{proof}
Let $k=1$. Since we have an adjoint pair of contravariant functors $\Psi, \Psi'$ it follows from the triangle identities (cf. Remark \ref{contrAdj}): 
If $\alpha_X$ is an isomorphism then also $\alpha'_{\Psi (X)}$ and if $\alpha'_Z$ is an isomorphism then also $\alpha_{\Psi'(Z)}$. 
Now, since $\mcM$ is faithfully balanced we have that $\Psi$ induces an equivalence $\mcP^{op} \cong \tilde{\mcM} =\Psi (h_{\mcP})$ by Lemma \ref{extYon}. It follows from the definition of $\Psi'$ and a right module version of Lemma \ref{diffdesc} that $\cogen^1 (\tilde{\mcM})= \{ Z \in \mcM- \modu_1 \mid \alpha'_Z \text{ isom}\} $.\\
The rest is a straightforward generalization of the proof of \cite{MS}, Lemma 2.9.
\end{proof}

%%%%%%%%%%%%%%%%%%%%%%%%%%%%%%%%%%
%%%%%%%%%%%%%%%%%%%%%%%%%%%%%%%%%%

\subsubsection{$\boxed{\gen_k}$}

We study $\gen_k(\mcM) =\gen_k^{\Modu-\mcP} (\mcM) \subset \Modu-\mcP$. We again give a different description of these categories using tensor products of $\mcM$-modules (cf. Lemma \ref{diffdescgen}). This is the main ingredient in the proof of the symmetry principle in the next subsection.

%%%%%%%%%%%%%%%%%%%%%%%%%%%%%%%%%

We have the covariant functor 
 \[
 \Phi \colon \Modu-\mcP \to \Modu-\mcM, \quad X \mapsto \Hom_{\Modu-\mcP} (-,X) \lvert_{\mcM}
 \]

and $\Phi \lvert_{\gen_k(\mcM)}\colon \gen_k(\mcM) \to \modu_k- \mcM$ is fully faithful. 
We have an induced covariant functor 
\[ 
\varepsilon = \Phi'\circ \Phi \colon \Modu-\mcP \to \Modu-\mcP, \quad X \mapsto \varepsilon_X
\]
defined for $P \in \mcP$ as 
\[ 
\varepsilon_X(P):= \Phi (X)\otimes_{\mcM}\Psi (h_P) 
\]
and a natural transformation $\varphi \colon \varepsilon \to \id_{\Modu-\mcP}$, for $X \in \Modu-\mcP$ this is given by a morphism $\varphi_X \colon \varepsilon_X \to X$ which is defined at $P \in \mcP$ via 
\[ 
\begin{aligned} 
\Hom_{\Modu-\mcP} (-,X)\lvert_{\mcM} \otimes_{\mcM} (\Hom_{\Modu-\mcP} (h_P, -)\lvert_{\mcM})  &\to \Hom_{\Modu-\mcP} (h_P, X)=X(P)\\
\underbrace{g \otimes f}_{ \in \Hom (M,X)\otimes_{\mathbb{Z}} \Hom (h_P, M)}  & \mapsto g\circ f
\end{aligned} 
\]

\begin{rem} $\Phi$ and is right adjoint functor of $\Phi'$  between abelian categories therefore $\Phi$ is left exact and $\Phi'$ is right exact, $\varphi$ is the counit of this adjunction. If $M \in \mcM$, then $\varphi_M$ is an isomorphism.  
\end{rem}

\begin{lem} \label{diffdescgen}
For $1\leq k \leq \infty$ we have 
\[
\begin{aligned} &\gen_k^{\Modu-\mcP}(\mcM) =\\
&\{ X\in \Modu -\mcP \mid \varphi_X \text{ isom. }, \Phi(X)\in \modu_k-\mcM, \;\Tor^i_{\mcM}(\Phi (X),\Psi (h_P))=0, 1\leq i < k, \forall P \in \mcP \}
\end{aligned}
\]
\end{lem}

\begin{proof}
Let $X \in \gen_k(\mcM)$, then there exists an exact sequence $M_k \to \cdots \to M_0 \to X \to 0$ such that $\Phi$ preserves its exactness, this implies $\Phi (X) \in \modu_k-\mcM$. Now, we apply $\varepsilon= \Phi'\Phi$ and consider the commutative diagram 
\[ 
\xymatrix{M_k \ar[r]& \cdots\ar[r] & M_0\ar[r] & X\ar[r] &0 \\
\varepsilon_{M_k} \ar[u]^{\varphi_{M_k}}\ar[r]& \cdots \ar[r] & \varepsilon_{M_0}\ar[u]^{\varphi_{M_0}}\ar[r] & \varepsilon_X \ar[u]^{\varphi_{X}}\ar[r]& 0 }
\]
Now, since $\Phi'$ is right exact and $\varphi_{M_i}$ is an isomorphism for $0\leq i\leq k$, we conclude that $\varphi_X$ is an isomorphism and the lower row is exact. This implies $\Tor^i_{\mcM} (\Phi (X),\Psi (h_P))=0, 1\leq i < k$. \\
Conversely, if we take $X\in \Modu-\mcP$ fulfilling the assumptions in the set bracket of the lemma. We can apply $\Phi'$ to the projective $k$-presentation of $\Phi (X)$, then we can find a diagram as before but this time we know from the assumptions that the bottom row is exact. Furthermore, since $\varphi_*$ is an isomorphism in all places of the diagram, we have that also the top row is exact. This implies $X \in \gen_k^{\Modu-\mcP}(\mcM)$. 
\end{proof}

%\begin{rem} (??) (but $\Phi (X)\in \modu_1-\mcM$ seems an extra condition?)
%The essential image of $\Phi\lvert_{\gen_1(\mcM)}$ can be described as the objects where the unit of the adjunction is an isomorphism (intersected with $\modu_1-\mcM$??). 
%\end{rem}

\subsection{The symmetry principle}

Now, we study these subcategories in more general exact categories. 
For an exact category $\mcE$ with enough projectives $\mcP$ and an exact category $\mcF$ with enough injectives $\mcI$, we consider the covariant, exact, fully faithful functors  
\[
\begin{aligned}
\mathbb{P}\colon \mcE & \to \modu_{\infty}-\mcP, \quad X \mapsto \Hom_{\mcE}(-,X)\lvert_{\mcP} \\
\mathbb{I}\colon \mcF^{\op} &\to \modu_{\infty}-\mcI^{op}, \quad X \mapsto \Hom_{\mcF} (X,-)\lvert_{\mcI^{op}}
\end{aligned}
\]
cf. \cite{E-master}, Prop. 2.2.1, Prop. 2.2.8

\begin{rem} \label{P(gen)}
For an additive category $\mcM$ of $\mcE$ (resp. of $\mcF$) we have: \[ 
\begin{aligned}
\mathbb{P} (\gen_k^{\mcE} (\mcM)) &= \Bild \mathbb{P}  \cap \gen_k^{\Modu-\mcP} (\mathbb{P} (\mcM )), \\
\quad \mathbb{I} ((\cogen^k_{\mcF}(\mcM))^{op}) &= \mathbb{I} (\gen_k^{\mcF^{op}} (\mcM^{op})) = \Bild \mathbb{I} \cap \gen_k^{\Modu-\mcI^{op}} (\mathbb{I} (\mcM^{op}))
\end{aligned}
\]
This follows from remark \ref{gen-deflationClosed} since $\mathbb{P}\colon \mcE\to \Bild \mathbb{P}$ is an equivalence of exact categories and $\Bild \mathbb{P}$ is deflation-closed in $\modu_{\infty}-\mcP$ and $\modu_{\infty}-\mcP$ is deflation-closed in $\Modu -\mcP$. The second statement follows by passing to the opposite category. 
\end{rem}

As before, let $\Phi \colon \mcE \to \Modu-\mcM, \Phi(X)= \Hom_{\mcE}(-,X)\lvert_{\mcM}$, $\Psi \colon \mcE \to \mcM-\Modu, \Psi (X)= \Hom_{\mcE} (X,-)\lvert_{\mcM}$. 
We have the immediate corollary:
\begin{cor} (of Lem. \ref{diffdescgen} and Rem. \ref{P(gen)})
(1) Let $\mcE$ be an exact category with enough projectives $\mcP$ and $\mcM$ a full additive subcategory. Then the following are equivalent: 
\begin{itemize}
    \item[(1)] $X \in \gen_k^{\mcE}(\mcM)$
    \item[(2)] $\Phi (X) \in \modu_k-\mcM$ and for every $P \in \mcP$: 
    \[ 
   \Phi (X) \otimes_{\mcM} \Psi (P) \to \Hom_{\mcE}(P,X), \; \; g\otimes f \mapsto g \circ f
    \] 
    is an isomorphism, $\Tor^i_{\mcM} (\Phi (X),\Psi (P) )=0 $, $1\leq i <k$.
\end{itemize}
(2) If $\mcE$ is an exact category with enough injectives $\mcI$ and $\mcM$ a full additive subcategory. Then the following are equivalent: 
\begin{itemize}
    \item[(1)] $X \in \cogen^k_{\mcE}(\mcM)$
    \item[(2)] $\Psi (X) \in \mcM-\modu_k$ and for every $I \in \mcI$: 
    \[ 
    \Phi (I) \otimes_{\mcM} \Psi (X)  \to \Hom_{\mcF}(X,I), \; \; g\otimes f \mapsto g \circ f
    \] 
    is an isomorphism, $\Tor^i_{\mcM} (\Phi (I), \Psi (X))=0 $, $1\leq i <k$.
\end{itemize}
\end{cor}

\begin{thm} \label{sym}(Symmetry principle). 
Let $\mcE$ be an exact category with enough projectives $\mcP$ and enough injectives $\mcI$ and $k \geq 1$. The following two statements are equivalent: 
\begin{itemize}
    \item[(1)] $\mcP \subset \cogen^k_{\mcE} (\mcM)$ and $\Phi(I)=\Hom_{\mcE} (-,I)\lvert_{\mcM} \quad \in \modu_k-\mcM$ for every $I \in \mcI$ 
    \item[(2)] $\mcI \subset \gen_k^{\mcE} (\mcM)$ and $\Psi(P)=\Hom_{\mcE} (P,-)\lvert_{\mcM}\quad \in \mcM-\modu_k$ for every $P \in \mcP$
\end{itemize}
\end{thm}

\begin{proof}
We consider $\mathbb{P}, \mathbb{I}$ as before defined for the category $\mcE$. 
Then, it is straight forward from the previous Lemma to see that (1) and (2) are both equivalent to for all $P \in \mcP, I \in \mcI$, $\Psi(P)\in \mcM-\modu_k, \Phi(I) \in \modu_k-\mcM$ and 
 \[ 
    \Phi (I) \otimes_{\mcM}  \Psi (P)   \to \Hom_{\mcE}(P,I), \; \; g\otimes f \mapsto g \circ f
    \] 
is an isomorphism, $\Tor^i_{\mcM} (\Phi (I) ,\Psi (P) )=0 $, $1\leq i <k$.
Therefore (1) and (2) are equivalent. 
\end{proof}

\section{Acknowledgement}
The author is supported by the Alexander von Humboldt-Stiftung in the framework of the Alexander von Humboldt Professorship endowed by the Federal Ministry of Education and Research. 
\bibliographystyle{amsalpha}
\bibliography{Sauter-FBInfunctor}

% \bib, bibdiv, biblist are defined by the amsrefs package.
\begin{bibdiv}
\begin{biblist}

\bib{ASoII}{article}{
      author={Auslander, M.},
      author={Solberg, O~.},
       title={Relative homology and representation theory. {II}. {R}elative
  cotilting theory},
        date={1993},
     journal={Comm. Algebra},
      volume={21},
      number={9},
       pages={3033\ndash 3079},
         url={https://doi.org/10.1080/00927879308824718},
}

\bib{E-master}{thesis}{
      author={Enomoto, Haruhisa},
       title={Relative auslander correspondence via exact categories},
        type={Masterthesis},
        date={2018},
}

\bib{MS}{article}{
      author={Ma, Biao},
      author={Sauter, Julia},
       title={On faithfully balanced modules, {F}-cotilting and {F}-{A}uslander
  algebras},
        date={2020},
        ISSN={0021-8693},
     journal={J. Algebra},
      volume={556},
       pages={1115\ndash 1164},
         url={https://doi.org/10.1016/j.jalgebra.2020.03.026},
      review={\MR{4089561}},
}

\bib{Y}{article}{
      author={Yoneda, Nobuo},
       title={On {E}xt and exact sequences},
        date={1960},
        ISSN={0368-2269},
     journal={J. Fac. Sci. Univ. Tokyo Sect. I},
      volume={8},
       pages={507\ndash 576 (1960)},
      review={\MR{225854}},
}

\end{biblist}
\end{bibdiv}

\end{document}